\documentclass[a4paper]{amsart}

\usepackage{amssymb}
\usepackage{amsmath}
\usepackage{enumerate}
\usepackage{mathtools}
\usepackage{bm}
\usepackage{amsthm}
\usepackage{cleveref}

\newcommand{\FF}{\mathbf{F}}
\newcommand{\CC}{\mathbf{C}}

\newcommand{\PP}{\mathbf{P}}

\newcommand{\ch}{\mathop{\mathrm{ch}}\nolimits}

\newcommand{\Sym}{\mathop{\mathrm{Sym}}\nolimits}

\newcommand{\Disc}{\mathop{\mathrm{Disc}}\nolimits}

\newcommand{\Cat}{\mathop{\mathrm{Cat}}\nolimits}
\newcommand{\rank}{\mathop{\mathrm{rank}}\nolimits}
\newcommand{\Gal}{\mathop{\mathrm{Gal}}\nolimits}

\newtheorem{thm}{Theorem}[section]
\newtheorem{prop}[thm]{Proposition}
\newtheorem{lem}[thm]{Lemma}
\newtheorem{conj}[thm]{Conjecture}

\newtheorem{rmk}[thm]{Remark}
\newtheorem{exa}[thm]{Example}

\title[Exponential sums over singular binary quintics]{Exponential sums over singular binary quintics}
\author{Yasuhiro Ishitsuka}
\date{\today}

\begin{document}

\begin{abstract}
    We give a characteristic-free explicit formula for exponential sums over singular binary quintic forms, based on the Waring decomposition of binary forms.  This extends the method used in our previous work on the space of binary quartics to a non-coregular space.
\end{abstract}
\maketitle

\section{Introduction}

In \cite{ITTX24}, the author and his coauthors determined the values of exponential sums over singular binary quartic forms over prime finite fields and applied the result to counting 2-Selmer elements with almost-prime discriminants. 
That work extends earlier results in \cite{Ishimoto}, \cite{Mor10}, and \cite{TT20a} on various \emph{prehomogeneous} vector spaces.  
In contrast to the preceding results, the space of binary quartic forms is not {prehomogeneous} but \emph{coregular}, 
and the authors had to employ another method to compute the exponential sums.
The purpose of this paper is to show that the method of \cite{ITTX24} applies to the space of binary quintics, which is not even \emph{coregular} (see \cite{Gey74} or \cite{GY}).

To describe our results, we recall the definition of the exponential sums of singular binary forms.  Let $p$ be a prime, and $q = p^k$ a power of $p$.
Let $V = V_1$ be a two-dimensional vector space over the field $\FF_q$ with $q$ elements.
Fix a nontrivial additive character $\psi \colon \FF_q \to \CC^\times$.
Let $V^*_n$ be the space of binary $n$-ic forms on $V_1$. 
The space $V_n^*$ admits a distinguished invariant $\Disc(f)$, the \emph{discriminant}, 
which detects the \emph{singular binary forms}, namely, the binary forms with a repeated factor.
The exponential sum of singular binary $n$-ic forms is defined as
\[
    \widehat{\Phi}(w) \coloneqq \dfrac{1}{\# V^*_n} \sum_{\substack{f \in V_n^* \\ \Disc(f) = 0}} \psi((w, f) ) \quad (w \in V_n).
\]
Here, the space $V_n$ is the dual of $V_n^*$, and $(\cdot, \cdot) \colon V_n \times V_n^* \to \FF_q$ is the natural perfect pairing.
In \cite{ITTX24}, we determined the sum for $n=4$, while \cite{TT20a} and \cite{Mor10} treat the case $n=3$.  
The method employed in \cite{ITTX24} can be described briefly as follows: we express the sum as a linear combination of counting functions for $\FF_q$-rational points on schemes 
and combine this expression with results from classical invariant theory.  

In this paper, we show that the method extends to the case 
$n=5$, where the \emph{Waring decomposition} is sufficient to obtain an explicit evaluation. Although the quintic case appears more complicated than the quartic case, the argument is in fact simpler, since it does not require the finer invariant-theoretic input needed when $n=4$.

\begin{thm}[{The main part of \Cref{thm:ExpQuint}}]\label{thm:Quint}
 	Let
    \[
        w = a_0(x^5)^* + a_1(x^4y)^* + \dots + a_5(y^5)^* \in V_5
    \]
    be an element.  
	If the covariant defined by
    \[
        \mathrm{Cat}_{2,2}(x, y \colon w) = \det \left(
        \begin{pmatrix}
            a_1 & a_2 & a_3 \\
            a_2 & a_3 & a_4 \\
            a_3 & a_4 & a_5
        \end{pmatrix} x
        -\begin{pmatrix}
            a_0 & a_1 & a_2 \\
            a_1 & a_2 & a_3 \\
            a_2 & a_3 & a_4
        \end{pmatrix}y
        \right)
    \]
    is nonzero,
	we have $\widehat{\Phi}(w) = C(w)q^{-4}$.
	Here $C(w)$ is the weighted sum 
    \[
       - \sum_{[s_0:s_1:s_2] \in \PP^2(\FF_q)}
       \varepsilon([s_0x^2 + s_1xy + s_2y^2])
    \]
    where the sum is taken over the points $[s_0: s_1: s_2] \in \PP^2(\FF_q)$ satisfying
    \[
        \begin{pmatrix}
            s_0 & s_1 & s_2
        \end{pmatrix}
        \begin{pmatrix}
            a_0 & a_1 & a_2 \\
            a_1 & a_2 & a_3 \\
            a_2 & a_3 & a_4
        \end{pmatrix}
        \begin{pmatrix}
            s_0 \\ s_1 \\ s_2
        \end{pmatrix} = \begin{pmatrix}
            s_0 & s_1 & s_2
        \end{pmatrix}
        \begin{pmatrix}
            a_1 & a_2 & a_3 \\
            a_2 & a_3 & a_4 \\
            a_3 & a_4 & a_5
        \end{pmatrix}
        \begin{pmatrix}
            s_0 \\ s_1 \\ s_2
        \end{pmatrix} = 0,
    \]
    and the function $\varepsilon \colon \PP(V_2^*) \to \CC$ is defined as
    \[
        \varepsilon([f_2]) = \begin{cases}
            1 & (\text{if $f_2$ is split}) \\
            0 & (\text{if $f_2$ is ramified}) \\
            -1 & (\text{if $f_2$ is inert})
        \end{cases}
    \]
 \end{thm}

Hence $C(w)$ is an integer with $|C(w)| \le 4$.
Moreover, this estimate is optimal in the sense that, as $\FF_q$ varies, $C(w)$ attains every integer value between $-4$ and $4$. We provide examples realizing each of these values.

\begin{rmk}
    In independent work, Wafik studies a related exponential sum in \cite{Waf}. He applies similar geometric and Waring decompositions to a different indicator function and expresses the numbers of rational points on the resulting auxiliary schemes in terms of various covariants. In particular, he explicitly determines the number of points in the fiber of the morphism associated with forms of the shape $f_2^2f_1$, one of the fiber-counting problems considered in Section \ref{sct:2^21}.
\end{rmk}

In view of the results for $n=3$ in \cite{TT20a} and $n=4$ in \cite{ITTX24}, we conjecture the following estimate for exponential sums in the generic case:

\begin{conj}\label{conj:est}
	Fix $n \ge 3$.
    For a generic form $w$ of degree $n$, we have $\widehat{\Phi}(w) = O(q^{-(n+3)/2})$.
\end{conj}

Note that this estimate for the singular exponential sum is better than the bound $O(q^{-(\dim V + 1)/2}) = O(q^{-(n+2)/2})$ suggested by square-root cancellation.
We will address this conjecture in a forthcoming paper.

This paper is organized as follows.
In Section \ref{sct:Waring}, we recall results on Waring decompositions in a form applicable over fields of small characteristic.  
In Section \ref{sct:Quintics}, we show Theorem \ref{thm:Quint} for the case when $\ch \FF_q > 2$.
The proof is based on the method of \cite{ITTX24}: we express the exponential sum in terms of counting functions for rational points on certain schemes and evaluate these functions using classical invariant theory.  The results in Section \ref{sct:Waring} are used in these calculations.
We also give examples of exponential sums achieving each integer $C(w)$ with $|C(w)| \le 4$.
In the appendix, we consider the case $\ch \FF_q = 2$.

\section*{Acknowledgements}
The author thanks Tatsuya Ohshita, Tetsushi Ito, Takashi Taniguchi, and Frank Thorne for their comments and discussions.  He also thanks Mohamed Wafik for sharing an unpublished manuscript and for helpful discussions on related counting problems for binary quintic forms.
He is supported by JSPS KAKENHI Grant Numbers 20K03747, 21K13773, 21K18557, 24K21512, and 26K06745.

The author acknowledges the use of OpenAI's ChatGPT and Codex during the preparation of this work. These tools were used as aids in organizing and refining the exposition and in checking calculations and other details. All mathematical arguments, proofs, and final verifications were independently reviewed by the author, who assumes full responsibility for the contents of this paper.

\section{Preliminaries on generalized Waring decomposition}\label{sct:Waring}
To obtain our main estimate (Theorem \ref{thm:Quint}), 
we need to recall some results on generalized Waring decompositions.
For our purposes, the results must be stated in a characteristic-free form.
For reference, see \cite{IK}.

\subsection{Exponential sums on singular binary forms}
Let $p$ be a prime, $q = p^k$ be a power of $p$ with $k \ge 1$,
and fix an additive character $\psi \colon \FF_q \to \CC^\times$.
Let $V_1$ be a two-dimensional vector space over $\FF_q$, with the basis $e_x, e_y$.
Let $V_1^*$ denote the dual space of $V_1$, and let $x,y$ be the basis dual to $e_x,e_y$.

For $n \ge 1$, we define $V_n^* = \Sym^n V_1^*$ the space of binary $n$-ic forms with coordinates $x, y$: for other $n$, we define
$V_0^* = \FF_q$, and $V_n^* = 0$ for $n < 0$. 
Let $V^*_{\bullet} = \bigoplus_{n \ge 0} V_n^*$ be the symmetric algebra of $V_1^*$.
It is isomorphic, as a graded algebra, to $\FF_q[x,y]$ with its standard grading.
An element of $V_n^*$ is \emph{singular} if it has a repeated factor.

We introduce the notion of splitting types.
Let us assume that a nonzero element $f \in V_n^*$ has an
irreducible factorization $cf_1^{e_1} f_2^{e_2} \dots f_r^{e_r}$ with a constant $c \in \FF_q$ and distinct factors $f_i$ of degree $d_i$.
Then we say that $f$ is of the {\emph{splitting type}} $\langle d_1^{e_1}, d_2^{e_2}, \dots, d_r^{e_r}\rangle$.  This is a combinatorial abstraction of irreducible factorizations.
The element is singular if and only if there is an index $1 \le i \le r$ with $e_i > 1$.

Let $V_n = D_n V = (V_n^*)^*$ be the dual space of $V_n^*$, and
$V_\bullet = \bigoplus_{n \ge 0} V_n$ be the divided coalgebra of $V_1^*$ with standard grading.
Let
\[
	( \cdot, \cdot) \colon V_n \times V_n^* \to \FF_q
\]
be the natural pairing.
We will consider the indicator function $\Phi \colon V_n^* \to \CC$ of singular binary forms as a function on $V_n^*$,
and the corresponding exponential sum $\widehat{\Phi} \colon V_n \to \CC$ defined by
\[
	\widehat{\Phi}(w) = \frac{1}{\# V_n^*} \sum_{v \in V_n^*} \Phi(v) \psi((w, v))
\]
as a function on $V_n$.

\subsection{Polar pairings}
For $0 \le m \le n$, let $\mu \colon V_{m}^* \times V_{n-m}^* \to V_{n}^*$ be multiplication in the symmetric algebra $V^*_\bullet$.
The \emph{polar pairing} 
\[
	( \cdot, \cdot ) \colon V_n \times V_m^* \to V_{n-m}
\]
is defined as 
\[
	( w, f_1 ) (f_2) = ( w, \mu(f_1, f_2) ) = (w, f_1f_2)
\]
for $w \in V_n$, $f_1 \in V_m^*$ and $f_2 \in V_{n-m}^*$. 
When $n = m$, it is the natural pairing $V_n \otimes V_n^* \to V_0 = \FF_q$, and the polar pairing is a natural extension of the pairing. 
We also note that when $n < m$, the pairing is the zero pairing since $V_{n-m} = 0$.

We describe the matrix representation of the polar pairing.
Let
\[
	\mathrm{st}^*_n = (x^n, x^{n-1}y, \dots, y^n)
\] 
be the standard basis of $V_n^*$, and define
\[
	\mathrm{st}_n = ((x^n)^*, (x^{n-1}y)^*, \dots, (y^n)^*)
\]
as the dual basis of $\mathrm{st}^*_n$.
Let $w \in V_n$ be an element described by
\[
	w = a_0(x^n)^* + a_1 (x^{n-1}y)^* + \dots + a_n (y^n)^*.
\]
We introduce the matrices
\[
	A_{s, t; r} = A_{s, t; r}(w) \coloneqq \begin{pmatrix}
		a_{r} & a_{r+1} & \dots & a_{r+s} \\
		a_{r+1} & a_{r+2} & \dots & a_{r+s+1} \\
		\vdots & \vdots & \ddots & \vdots \\
		a_{r+t} & a_{r+t+1} & \dots & a_{r+s+t}
	\end{pmatrix}
\]
for $0 \le r, s, t$ with $r + s + t \le n$.
This represents the map 
\[
	(w, x^{n-r-s-t}y^{r})(\cdot) 
    \colon V_s^* \to V_t
\]
with respect to the bases $\mathrm{st}_s^*$ and $\mathrm{st}_t$.

If $(w, f)= 0$, we say $w$ is \emph{apolar} to $f$, and 
$f$ is \emph{apolar} to $w$.
For fixed $w \in V_n$, we denote the set of polynomials apolar to $w$ 
\[
	\begin{aligned}
		w^\bot &\coloneqq \{ f \in V^*_\bullet \mid ( w, f ) = 0\}, \\
		(w^\bot)_n &\coloneqq \{ f \in V^*_n \mid ( w, f ) = 0\}.
	\end{aligned}
\]
By definition, the set $w^\bot$ forms an ideal of $V^*_\bullet$, called the \emph{apolar ideal} to $w$.

\begin{thm}\label{thm:CI}
	Let $0 \neq w \in V_n$ with $n =2m+1 > 0$.
	\begin{enumerate}[{(i) }]
		\item (\cite[Theorem 1.54]{IK}.) Let $s \coloneqq \rank A_{m+1,m; 0}(w) \le m+1$. Then the apolar ideal $w^\bot$ is a complete intersection ideal of $V^*_\bullet$ 
		generated by two elements $\varphi_s \in V_s^*$ and $\varphi_{n+2-s} \in V_{n+2-s}^*$ (note that $s < n+2-s$).
		\item ({cf. \cite[Lemma 5.3]{KR}}) The covariant
		\[
			\Cat_{m, m}(x, y; w) = \det \left(
				A_{m, m; 1}(w) x - A_{m, m; 0}(w) y
			\right) \in V_{m+1}^*
		\]
		is apolar to $w$.  If $\Cat_{m,m}(x,y; w)$ vanishes, then the integer $s$ in (i) is less than $m+1$.
	\end{enumerate}
\end{thm}
\begin{proof}
	The first part (i) is \cite[Theorem 1.53 (i), Theorem 1.54]{IK}.  We prove (ii).
	Since the linear map $(w, \cdot) \colon V_{m+1}^* \to V_{m}$ is represented by
	\[
		A_{m+1, m;0}(w) = \begin{pmatrix}
			a_{0} & a_{1} & \dots & a_{m+1} \\
			a_{1} & a_{2} & \dots & a_{m+2} \\
			\vdots & \vdots & \ddots & \vdots \\
			a_{m} & a_{m+1} & \dots & a_{2m+1}
		\end{pmatrix}
	\]
    with respect to the bases $\mathrm{st}_{m+1}^*$ and $\mathrm{st}_m$,
	the binary form of degree $m+1$
	\[
        \det \begin{pmatrix}
            x^{m+1} & x^{m}y & \dots & y^{m+1}\\
			a_{0} & a_{1} & \dots & a_{m+1} \\
			a_{1} & a_{2} & \dots & a_{m+2} \\
			\vdots & \vdots & \ddots & \vdots \\
		      a_{m} & a_{m+1} & \dots & a_{2m+1} 
		\end{pmatrix},
	\]
	is in the kernel of the map $(w, \cdot)$. In other words, the form is apolar to $w$.
	By direct calculation, 
    we may check that this form coincides with $\Cat_{m,m}(x,y; w)$.

	The coefficients of $\Cat_{m,m}(x,y;w)$ are, up to signs, the maximal minors of $A_{m+1, m;0}(w)$. 
    Hence if $\Cat_{m,m}(x,y; w) = 0$, any maximal minor of $A_{m+1, m; 0}(w)$ vanishes, and $A_{m+1, m; 0}(w)$ has rank at most $m$.
\end{proof}

We denote the integer $1 \le s \le \lceil n/2 \rceil$ in Theorem \ref{thm:CI}(i) by $\rank(w)$ and call it the \emph{catalecticant rank}, following \cite{BBM}.
We define $\rank(0) = 0$.
By Theorem \ref{thm:CI} (ii), if $w \neq 0$ is of odd degree, 
the minimal generator $\phi_s$ of degree $s$ is unique up to scalar.
We define the \emph{Waring type} of $w$ of odd degree as the splitting type of 
the minimal generator $\phi_s$.

\subsection{Generalized Waring decomposition}

To compute the exponential sum, we use a more explicit description of $w$ in terms of its Waring type.
For a binary form $f \in V_n^*$, we define the set of elements in $V_\bullet$ apolar to $f \in V^*_\bullet$:
\[
	\begin{aligned}
		^{\bot}f &\coloneqq \left\{ w \in V_\bullet \mid (w, f) = 0 \right\}, \\
		(^{\bot}f)_r &\coloneqq \left\{ w \in V_r \mid (w, f)= 0 \right\}.
	\end{aligned}
\]
\begin{exa}
As is easily checked, for a nonzero linear form $\ell(x,y) = ax + by$, the space $({}^\bot \ell)_n$ is spanned by the single element
\[
	b^n (x^n)^* - ab^{n-1} (x^{n-1}y)^* + \dots + (-1)^n a^n (y^n)^*. 
\]
Moreover, we can describe the case when $\ell = x$ or $y$:
\[
	\begin{aligned}
		({}^\bot x^e)_n &= \{a_0 (x^n)^* + a_1(x^{n-1}y)^* + \dots + a_n(y^n)^* \in V_n \mid a_0 = \dots = a_{n-e} = 0\} \\
        ({}^\bot y^e)_n &= \{a_0 (x^n)^* + a_1(x^{n-1}y)^* + \dots + a_n(y^n)^* \in V_n \mid a_e = \dots = a_{n} = 0 \}.
	\end{aligned}
\]
\end{exa}
Assume that $f$ factors as $f_1^{e_1}f_2^{e_2} \dots f_r^{e_r}$, where the $f_i$ are pairwise coprime irreducible factors.
Let $\langle h_1, h_2, \dots, h_u \rangle \subset V_\bullet^*$ be the ideal of $V^*_\bullet$ generated by $h_1, \dots, h_u \in V^*_\bullet$, and define
$\langle h_1, h_2, \dots,h_u \rangle_m$ as the degree $m$ part of the ideal $\langle h_1, h_2, \dots, h_u \rangle$.
Then we have
\[
	\langle f \rangle = \bigcap_{1 \le i \le r} \langle f_i^{e_i} \rangle, \quad \text{and} \quad \langle f \rangle_m = \bigcap_{1 \le i \le r} \langle f_i^{e_i} \rangle_m.
\]
Taking orthogonal complements in the perfect pairing $V_m \times V_m^* \to V_0 = \FF_q$, the second equality above implies that, for every integer $m \ge 0$,
\begin{equation}\label{eq:sumofVS}
	({}^\bot f)_m = \sum_{1 \le i \le r} ({}^\bot f_i^{e_i})_m.
\end{equation}
Moreover, if $n \le m$, then we have the following theorem.
When $\ch \FF_q = p > n$, this result is well known (see \cite[Lemma 1.31]{IK}, \cite{GY}, and \cite[Lemma 3.3]{KR}). 

\begin{thm}[{Generalized Waring decomposition}]\label{thm:Waring}
	Let $k$ be a field of any characteristic.
	Let $f \in V^*_n$ be a binary $n$-ic form over $k$, and $n \le m$.
    Then we have
	\[
	   ({}^\bot f)_m = \bigoplus_{1 \le i \le r} ({}^\bot f_i^{e_i})_m.
    \]
\end{thm}
\begin{proof}
The space $({}^\bot f)_m$ is explicitly described as
\[
	({}^\bot f)_m = \{ w \in V_m \mid (w, x^{m-n}f) = (w, x^{m-n-1}y f) = \dots = (w, y^{m-n}f) = 0\}.
\]
Since the pairing $(\cdot, \cdot) \colon V_m \times V^*_m \to \FF_q$ is perfect and $n \le m$, we have
\[
	\dim ({}^\bot f)_m = n = \deg f.
\]
Because similar arguments can be applied to $f_i^{e_i} (1 \le i \le r)$,
we see that
\[
	\dim ({}^\bot f)_m = \deg f = \sum_{1 \le i \le r} \deg f_i^{e_i} = \sum_{1 \le i \le r} \dim ({}^\bot f_i^{e_i})_m,
\]
and the right hand side of \eqref{eq:sumofVS} is the direct sum.
\end{proof}

\subsection{Waring types on binary quintics}
We apply the results in the previous subsections to binary quintics.
Let $\tau \in \Gal(\overline{\FF_q}/\FF_q)$ be the Frobenius automorphism.
By abuse of notation, we also use $\tau$ to denote the action 
$1 \otimes \tau \colon V^*_\bullet \otimes \overline{\FF_q} \to V^*_\bullet \otimes \overline{\FF_q}$.
We classify the elements $w \in V_5$ by their Waring types and catalecticant ranks.
\begin{enumerate}[{(1) }]
	\item Catalecticant rank 0: the unique Waring type is $\emptyset$, which implies that $w=0$.
	\item Catalecticant rank 1: the unique Waring type is $\langle 1 \rangle$.
	In this case, there is a linear form $f_1 \in V^*_1$ apolar to $w$; equivalently, $w \in ({}^\bot f_1)_5$.
    \item Catalecticant rank 2: There is a quadric $f \in (w^\bot)_2$ unique up to scalar. There are three subcases:
    \begin{enumerate}[{(a) }]
	   \item Waring type $\langle 1^2 \rangle$.  We may assume that $f = f_1^2$ for a linear form $f_1 \in V^*_1$, and $w \in ({}^\bot f_1^2)_5$.
	   \item Waring type $\langle 1, 1 \rangle$. There are linearly independent linear forms $f_{1,1}, f_{1,2} \in V^*_1$ with $f = f_{1,1}f_{1,2}$.
	   The element is $w \in ({}^\bot f_{1,1}f_{1,2})_5 = ({}^\bot f_{1,1})_5 + ({}^\bot f_{1,2})_5$.
	   \item Waring type $\langle 2 \rangle$.  
	   There is a linear form $f_1 \in V^*_1 \otimes \FF_{q^2} \setminus V^*_1$ with $f = f_1 \tau(f_1)$.
	   The element is $w \in ({}^\bot f_1 \tau(f_1))_5 = ({}^\bot f_1)_5 + ({}^\bot \tau(f_1))_5$.
    \end{enumerate}
    \item Catalecticant rank 3: There is a nonzero cubic $f_3 \in (w^\bot)_3$ unique up to scalar.
    By Theorem \ref{thm:CI}, we may assume that $f_3 = \Cat_{2,2}(x,y; w)$, and 
    the Waring type of $w$ is the splitting type of $\Cat_{2,2}(x,y; w)$.  There are five subcases:
	\begin{enumerate}[{(a) }]
        \item Waring type $\langle 1^3 \rangle$. 
		We may assume that $f_3 = f_1^3$ for a linear form $f_1 \in V^*_1$, and $w \in ({}^\bot f_1^3)_5$.
	\item Waring type $\langle 1^2, 1 \rangle$.  We can write $f_3 = f_{1,1}^2f_{1,2}$ for linearly independent linear forms
		$f_{1,1}, f_{1,2} \in V_1^*$, and $w \in ({}^\bot f_{1,1}^2)_5 + ({}^\bot f_{1,2})_5$.
	\item Waring type $\langle 1, 1, 1 \rangle$.  We can write $f_3 = f_{1,1} f_{1,2} f_{1,3}$ 
		for pairwise linearly independent $f_{1,1}, f_{1,2}, f_{1,3} \in V^*_1$.
		Then $w \in ({}^\bot f_{1,1})_5 + ({}^\bot f_{1,2})_5 + ({}^\bot f_{1,3})_5$.
        \item Waring type $\langle 2, 1 \rangle$.  
        \item Waring type $\langle 3 \rangle$.  
    \end{enumerate}
\end{enumerate}

\section{Exponential sums on singular binary quintics}\label{sct:Quintics}
In this section, we apply the results of the previous section to binary quintics
and obtain our main result.  Throughout this section, we assume that $\ch \FF_q = p > 2$; the characteristic-two case is treated in the appendix.

\subsection{Contraction of the exponential sum}

As in \cite{ITTX24}, we reduce the calculation of the sum to counting $\FF_q$-rational points on schemes (see \cite[pp.\ 136--137]{FK}).
Since $\Phi$ is invariant under scalar multiplication, we divide the sum into the terms with $f = 0$ and $f \neq 0$, 
and then decompose the latter as a sum over $[f] \in \PP(V^*_n)$ and $f \in [f]$:

\[
	\begin{aligned}
		&\phantom{=} \sum_{f \in V_n^*} \Phi(f) \psi((w,f))
		= 1 + \sum_{[f] \in \PP(V_n^*)} \Phi([f]) \sum_{g \in [f]} \psi((w,g)).
	\end{aligned}
\]
Since for $0 \neq f \in V_n^*$,
\[
	\sum_{g \in [f]} \psi((w,g)) = \begin{cases}
		q-1 & (\text{if $(w,f) = 0$}) \\
		-1 & (\text{if $(w, f) \neq 0$}),
	\end{cases}
\]
we have
\[
	\begin{aligned}
		&\phantom{=} 1 + \sum_{[f] \in \PP(V_n^*)} \Phi([f]) \sum_{g \in [f]} \psi((w, g)) \\
		&= 1 + q \sum_{\substack{[f] \in \PP(V_n^*) \\ \Disc(f) = (w, f) = 0}}1 - \sum_{\substack{[f] \in \PP(V_n^*) \\ \Disc(f) = 0}} 1.
	\end{aligned}
\]
Putting
\[
	N_w \coloneqq \# \{[f] \in \PP(V_n^*) \mid \Disc(f) = (w, f) = 0\}
\]
for $w \in V_n$, we obtain
\[
	q^6\widehat{\Phi}(w) = 1 + qN_w - N_0.
\]
Thus we concentrate on the calculation of $N_w$ for $w \in V_5$.

\subsection{Geometric decomposition in the case of binary quintics}
The three relevant morphisms are
\[
    \begin{aligned}
        &\psi_{1^2, 3} \colon \PP(V_1^*) \times \PP(V_3^*) \to \PP(V_5^*) & &;& & ([f_1], [f_3]) \mapsto [f_1^2 f_3], \\
        &\psi_{2^2, 1} \colon \PP(V_2^*) \times \PP(V_1^*) \to \PP(V_5^*) & &;& &([f_2], [f_1]) \mapsto [f_2^2 f_1], \\
        &\psi_{1^2, 1^2, 1} \colon \PP(V_1^*) \times \PP(V_1^*) \times \PP(V_1^*) \to \PP(V_5^*) & &;& &([f_{1,1}], [f_{1,2}], [f_{1,3}]) \mapsto [f_{1,1}^2 f_{1,2}^2 f_{1,3}].
    \end{aligned}
\]
As in the quartic case \cite[(15)]{ITTX24}, we have the following:
\begin{prop}
For each $[f] \in \PP(V_5^*)$, we have
\[
   \bm{1}_{\Disc}([f]) = \# \psi_{1^2, 3}^{-1}([f]) - \# \psi_{1^2, 1^2, 1}^{-1}([f]) + \# \psi_{2^2, 1}^{-1}([f]).
\]
\end{prop}
\begin{proof}
    Note that each term only depends on the splitting type of $f$.
    The following table proves the formula directly.
    \begin{table}[h]
        \begin{tabular}{c||ccc|c}
            \hline
            splitting type of $[f]$ & $\#\psi^{-1}_{1^2, 3}([f])$ & $\#\psi^{-1}_{1^2, 1^2, 1}([f])$ & $\#\psi^{-1}_{2^2, 1}([f])$ & $\bm{1}_{\Disc}([f])$ \\
            \hline
            nonsingular & $0$ & $0$ & $0$ & $0$ \\
            \hline
            $\langle 1^2, 1, 1, 1\rangle$ & $1$ & $0$ & $0$ & $1$ \\ 
            $\langle 1^2, 2, 1\rangle$ & $1$ & $0$ & $0$ & $1$ \\ 
            $\langle 1^2, 3\rangle$ & $1$ & $0$ & $0$ & $1$ \\ 
            $\langle 1^3, 1, 1\rangle$ & $1$ & $0$ & $0$ & $1$ \\ 
            $\langle 1^3, 2\rangle$ & $1$ & $0$ & $0$ & $1$ \\ 
            \hline
            $\langle 1^2, 1^2, 1\rangle$ & $2$ & $2$ & $1$ & $1$ \\ 
            $\langle 1^3, 1^2\rangle$ & $2$ & $2$ & $1$ & $1$ \\ 
            \hline
            $\langle 2^2, 1\rangle$ & $0$ & $0$ & $1$ & $1$ \\ 
            \hline
            $\langle 1^4, 1\rangle$ & $1$ & $1$ & $1$ & $1$ \\ 
            $\langle 1^5\rangle$ & $1$ & $1$ & $1$ & $1$ \\ 
            \hline
        \end{tabular}
    \end{table}
\end{proof}
For simplicity, we write $[w^\bot]_5 \coloneqq \PP((w^\bot)_5) \subseteq \PP(V^*_5)$.
Summing up over $[f] \in [w^\bot]_5$,  we have the following formula:
\[
   N_w= \# \psi_{1^2, 3}^{-1}([w^\bot]_5) - \# \psi_{1^2, 1^2, 1}^{-1}([w^\bot]_5) + \# \psi^{-1}_{2^2, 1}([w^\bot]_5).
\]

\subsection{Preparation for counting the fibers of the three morphisms}
For the computation, we restate each term in terms of catalecticant matrices.
To do this, let us write
\(
	w = a_0(x^5)^* + a_1 (x^{4}y)^* + \dots + a_5 (y^5)^* \in V_5.
\)

\subsubsection{The morphism $\psi_{1^2, 3}$}
Fix coordinates as
\[
	\begin{aligned}
		f_1 &= s_0x + s_1y, \\
		f_3 &= t_0x^3 + t_1 x^2y + t_2 xy^2 + t_3 y^3.
	\end{aligned}
\]
Then the condition $(w, f_1^2 f_3) = 0$ requires a linear condition on $(t_0, t_1, t_2, t_3)$,
and the requirement is nontrivial unless
\begin{equation}\label{eq:113}
	(w, x^3 f_1^2) = (w, x^2y f_1^2) = (w, xy^2 f_1^2) = (w, y^3 f_1^2) = 0.
\end{equation}
The condition \eqref{eq:113} is equivalent to two conditions: (i) $w \in {}^\bot f_1^2$, or (ii) $f_1^2 \in w^\bot$.
\begin{lem}\label{lem:113}
    If $w \neq 0$ and there is a linear form $f_1$ with $f_1^2 \in w^\bot$, then
    it is unique up to scalar. 
    
    Moreover, such a linear form exists in exactly two cases:
    (1) $\rank(w) = 1$, or (2) $w$ has Waring type $\langle 1^2 \rangle$.
\end{lem}
\begin{proof}
    By Theorem \ref{thm:CI}, when $w \neq 0$, the apolar ideal $w^\bot \subseteq \FF_q[x,y]$ is 
    the complete intersection ideal generated by $\varphi_{\rank(w)}$ of degree $\rank(w) \le 3$ 
    and $\varphi_{7 - \rank(w) }$ of degree $7-\rank(w)$.
    
    When $\rank(w) = 3$, we have $(w^\bot)_2 = 0$, so there is no nonzero $f_1$ with $f_1^2 \in w^\bot$.  When $\rank(w) = 2$, such a linear form $f_1$ 
    satisfies $f_1^2 \in \langle \varphi_{2}, \varphi_{5} \rangle$, which implies that $f_1^2$ and $\varphi_{2}$ coincide up to scalar.  This is case (2).
    When $\rank(w) = 1$, we find that $f_1^2$ is divisible by $\varphi_1$, and hence $f_1$ coincides with $\varphi_1$ up to scalar.  This is case (1).
\end{proof}

The number $\#\psi^{-1}_{1^2, 3}([w^\bot]_5)$ is calculated as follows: 
\begin{enumerate}[{(i) }]
    \item For any $w \in V_5$ and $[f_1] \in \PP(V_1^*)$, there are at least $q^2 + q + 1 = \#\PP^2(\FF_q)$ choices for $[f_3] \in \PP(V_3^*)$.  
    \item For each $[f_1] \in \PP(V_1^*)$ with $f_1^2 \in w^\bot$, there are an additional $q^3 = \#\PP^3(\FF_q) - \# \PP^2(\FF_q)$ choices for $[f_3] \in \PP(V_3^*)$.
\end{enumerate}
By Lemma \ref{lem:113}, when $w \neq 0$, the linear form $f_1 \in V_1^*$ with $f_1^2 \in w^\bot$ is unique if it exists.  When $w=0$, 
any linear form $f_1 \in V_1^*$ satisfies $f_1^2 \in (0^\bot)_2 = V^*_2$.
When we define
\[
    \begin{aligned}
    V_{1^2, 3:w} &\coloneqq \{[f_1] \in \PP(V^*_1) \mid f_1^2 \in w^\bot\}, \\
    \widetilde{N}_{1^2, 3:w} &\coloneqq \# V_{1^2, 3:w},
    \end{aligned}
\]
our arguments show the following.

\begin{prop}\label{prop:First}
    We have
	\[
		\# \psi_{1^2, 3}^{-1}([w^\bot]_5) = (q+1)(q^2 + q + 1) + q^3\widetilde{N}_{1^2, 3:w},
    \]
    where the last term is
    \[
        \widetilde{N}_{1^2, 3:w} = \begin{cases}
			q+1 & (w = 0) \\
			1 & (\text{$0 \neq w \in {}^\bot f_1^2$ for some $f_1 \in V^*_1$}) \\
			0 & (\text{otherwise}).
		\end{cases}
	\]
\end{prop}

\subsubsection{The morphism $\psi_{2^2, 1}$}\label{sct:2^21}
Fix coordinates as
\[
	\begin{aligned}
		f_2 &= s_0x^2 + s_1xy + s_2y^2, \\
		f_1 &= t_0x + t_1y.
	\end{aligned}
\]
Then the condition $(w, f_2^2 f_1) = 0$ requires a linear condition on $(t_0, t_1)$,
and the requirement is nontrivial unless
\begin{equation}\label{eq:221}
	(w, x f_2^2) = (w, yf_2^2) = 0.
\end{equation}
This is equivalent to (1) $w \in {}^\bot f_2^2$, or (2) $f_2^2 \in w^\bot$.
In terms of catalecticant matrices, it is stated as
\begin{equation}\label{eq:21}
	\begin{aligned}
	&\begin{pmatrix}
		s_0 & s_1 & s_2
	\end{pmatrix}
	\begin{pmatrix}
		a_0 & a_1 & a_2 \\
		a_1 & a_2 & a_3 \\
		a_2 & a_3 & a_4
	\end{pmatrix}
	\begin{pmatrix}
		s_0 \\ s_1 \\ s_2
	\end{pmatrix} \\
	= 
	&\begin{pmatrix}
		s_0 & s_1 & s_2
	\end{pmatrix}
	\begin{pmatrix}
		a_1 & a_2 & a_3 \\
		a_2 & a_3 & a_4 \\
		a_3 & a_4 & a_5
	\end{pmatrix}
	\begin{pmatrix}
		s_0 \\ s_1 \\ s_2
	\end{pmatrix} = 0
	\end{aligned}
\end{equation}
The number $\#\psi^{-1}_{2^2, 1}([w^\bot]_5)$ is calculated as follows: 
\begin{enumerate}[{(i) }]
    \item For any $w$ and $[f_2] \in \PP(V_2^*)$, there is one choice
for $[f_1] \in \PP(V_1^*)$.  
    \item For each $f_2$ with $w \in {}^\bot f_2^2$, there are an additional $q$ choices for $[f_1] \in \PP(V^*_1)$.
\end{enumerate}
Let $V_{2^2, 1:w}, \widetilde{N}_{2^2, 1: w}$ be 
\[
    \begin{aligned}
    V_{2^2, 1:w} &\coloneqq \{[f_2] \in \PP(V^*_2) \mid f_2^2 \in w^\bot\}, \\
    \widetilde{N}_{2^2, 1:w} &\coloneqq \# V_{2^2, 1:w}.
    \end{aligned}
\]
We compute
\[
	\# \psi_{2^2, 1}^{-1}([w^\bot]_5 ) = (q^2 + q + 1) + q \widetilde{N}_{2^2, 1: w}.
\]

\subsubsection{The morphism $\psi_{1^2, 1^2, 1}$}
Fix coordinates as
\[
	\begin{aligned}
		f_{1,1} &= s_0x + s_1y, \\
		f_{1,2} &= t_0x + t_1y, \\
		f_{1,3} &= u_0x + u_1y.
	\end{aligned}
\]
Then the condition $(w, f_{1,1}^2 f_{1,2}^2 f_{1,3}) = 0$ requires a linear condition on $(u_0, u_1)$,
and the requirement is nontrivial unless
\[
	(w, x f_{1,1}^2 f_{1,2}^2) = (w, yf_{1,1}^2 f_{1,2}^2) = 0,
\]
or equivalently, $w \in {}^\bot f_{1,1}^2 f_{1,2}^2$ or $f_{1,1}^2 f_{1,2}^2 \in w^\bot$. 
In terms of catalecticant matrices, it is stated as
\begin{equation}\label{eq:111}
	\begin{aligned}
	&\begin{pmatrix}
		s_0^2 & 2s_0s_1 & s_1^2
	\end{pmatrix}
	\begin{pmatrix}
		a_0 & a_1 & a_2 \\
		a_1 & a_2 & a_3 \\
		a_2 & a_3 & a_4
	\end{pmatrix}
	\begin{pmatrix}
		t_0^2 \\ 2t_0t_1 \\ t_1^2
	\end{pmatrix} \\
	= 
	&\begin{pmatrix}
		s_0^2 & 2s_0s_1 & s_1^2
	\end{pmatrix}
	\begin{pmatrix}
		a_1 & a_2 & a_3 \\
		a_2 & a_3 & a_4 \\
		a_3 & a_4 & a_5
	\end{pmatrix}
	\begin{pmatrix}
		t_0^2 \\ 2t_0t_1 \\ t_1^2
	\end{pmatrix} = 0
	\end{aligned}
\end{equation}
Let $V_{1^2, 1^2, 1:w}, \widetilde{N}_{1^2, 1^2, 1: w}$ be 
\[
    \begin{aligned}
    V_{1^2, 1^2, 1:w} &\coloneqq \{([f_{1,1}], [f_{1,2}]) \in \PP(V^*_1) \times \PP(V^*_1) \mid f_{1,1}^2f_{1,2}^2 \in w^\bot\}, \\
    \widetilde{N}_{1^2, 1^2, 1:w} &\coloneqq \# V_{1^2, 1^2, 1:w}.
    \end{aligned}
\]
Then we compute
\[
	\# \psi_{1^2, 1^2, 1}^{-1}([w^\bot]_5) = (q+ 1)^2 + q\widetilde{N}_{1^2,1^2,1:w}.
\]

\subsubsection{Difference of $\widetilde{N}_{2^2,1:w}$ and $\widetilde{N}_{1^2,1^2,1:w}$}
In the exponential sum, we compute the difference
\[
    \# \psi_{2^2, 1}^{-1}([w^\bot]_5) - \# \psi_{1^2, 1^2, 1}^{-1}([w^\bot]_5) = q(\widetilde{N}_{2^2,1:w} - \widetilde{N}_{1^2,1^2,1:w} -1).
\]
We can rewrite $\widetilde{N}_{2^2,1:w} - \widetilde{N}_{1^2,1^2,1:w}$ by a weighted count.

Recall the function $\varepsilon \colon \PP(V_2^*) \to \CC$ from the Introduction.
The number of elements in the fiber of the product map $\psi_{1,1} \colon \PP(V^*_1) \times \PP(V^*_1) \to \PP(V^*_2)$ is described by
\[
    \begin{aligned}
    \# \psi^{-1}_{1, 1}([f_2]) &= \#\{([f_{1,1}], [f_{1,2}]) \mid f_2 = f_{1,1}f_{1,2}\} \\
    &=
     1 +\varepsilon(f_2).
    \end{aligned}
\]
This gives us the following description:
\begin{equation}\label{eq:Sum}
    \begin{aligned}
         \widetilde{N}_{2^2,1:w} - \widetilde{N}_{1^2,1^2,1:w}
        &= \sum_{\substack{[f_2] \in \PP(V^*_2) \\ f_2^2 \in w^\bot}} 1 - 
        \sum_{\substack{[f_2] \in \PP(V_2^*) \\ f_2^2 \in w^\bot}} \left( 1 + \varepsilon(f_2)\right) \\
        &= - \sum_{\substack{[f_2] \in \PP(V^*_2) \\ f_2^2 \in w^\bot}} \varepsilon(f_2).
    \end{aligned}
\end{equation}
We will use this description to obtain the estimate in Theorem \ref{thm:Quint}.

\subsection{The case $w=0$}
The equations \eqref{eq:21} and \eqref{eq:111} pose no condition on
$[f_2] \in \PP(V^*_2)$ and $([f_{1,1}], [f_{1,2}]) \in \PP(V^*_1) \times \PP(V^*_1)$.
We have
\[
	\begin{aligned}
		\# \psi_{2^2, 1}^{-1}([0^\bot]_5) &= (q^2 + q + 1)(q+1), \\
		\# \psi_{1^2, 1^2, 1}^{-1}([0^\bot]_5) &= (q+1)^3.
	\end{aligned}
\]
This also implies
\[
    \begin{aligned}
        \widetilde{N}_{2^2,1:w} &= q^2 + q + 1, \\
        \widetilde{N}_{1^2,1^2,1:w} &= (q+1)^2.
    \end{aligned}
\]

\subsection{Catalecticant rank 1 case}
In this case, the ideal $w^\bot$ is generated by a linear form $f_1 \in V^*_1$ and
a binary sextic form.
This implies 
\(
    (w^\bot)_4 = f_1 V^*_3,
\)
and 
\[
    \{g_2 \in V^*_2 \mid g_2^2 \in w^\bot\} = f_1 V^*_1.
\]
In the space, there is only one quadric of type $\langle 1^2 \rangle$, $q$ quadrics of type $\langle 1, 1 \rangle$ and no quadrics of type $\langle 2 \rangle$.
This concludes
\[
	\begin{aligned}
		\widetilde{N}_{2^2, 1: w} &= q + 1, \\
		\widetilde{N}_{1^2, 1^2, 1: w} &= 2q + 1.
	\end{aligned}
\]

\subsection{Catalecticant rank 2 cases}
In these cases, the ideal $w^\bot$ is generated by a quadratic form $f_2 \in V^*_2$ and
a binary quintic form.  This implies 
\(
    (w^\bot)_4 = f_2 V^*_2.
\)

\subsubsection{The case of Waring type $\langle 1^2 \rangle$}
In this case, we may assume that $f_2 = f_1^2$ for a form $f_1 \in V^*_1$,
and then
\[
    \{g \in V^*_2 \mid g^2 \in w^\bot\} = f_1 V^*_1.
\]
We have the same conclusion as in the catalecticant-rank-one case:
\[
	\begin{aligned}
		\widetilde{N}_{2^2, 1: w}&= q + 1, \\
		\widetilde{N}_{1^2, 1^2, 1:w} &= 2q + 1.
	\end{aligned}
\]

\subsubsection{The case of Waring type $\langle 1, 1 \rangle$}
There are linearly independent linear forms $f_{1,1}, f_{1,2} \in V^*_1$ 
with $f_2 = f_{1,1}f_{1,2}$.  This shows
\[
    \{g \in V^*_2 \mid g^2 \in w^\bot\} = \FF_q f_{1,1} f_{1,2}.
\]
We have 
\[
	\begin{aligned}
		\widetilde{N}_{2^2, 1: w} &= 1, \\
		\widetilde{N}_{1^2, 1^2, 1: w} &= 2.
	\end{aligned}
\]

\subsubsection{The case of Waring type $\langle 2 \rangle$}
In this case the quadric $f_2 \in V^*_2$ is irreducible.
This shows
\[
    \{g \in V^*_2 \mid g^2 \in w^\bot\} = \FF_q f_2,
\]
and we have 
\[
	\begin{aligned}
		\widetilde{N}_{2^2, 1: w} &= 1, \\
		\widetilde{N}_{1^2, 1^2, 1: w} &= 0.
	\end{aligned}
\]

\subsection{Catalecticant rank 3 cases}
In these cases, it is enough to estimate the difference 
\(
    \widetilde{N}_{2^2, 1:w} - \widetilde{N}_{1^2, 1^2, 1:w}.
\)

\subsubsection{The case of Waring type $\langle 1^3 \rangle$}\label{sect:1^3}
In this case, we have $\Cat_{2,2}(x,y;w) = c f_1^3$ for a constant $c \in \FF_q^\times$ and a linear form $f_1 \in V_1^*$.
After a change of coordinates sending $f_1$ to $y$, we may assume $w \in {}^\bot y^3$.
Then we can write
\[
	w(x,y) = a_0(x^5)^* + a_1(x^4y)^*+ a_2(x^3y^2)^*
\]
with $a_2 \neq 0$.
The equation \eqref{eq:21} becomes
\begin{equation}\label{eq:21triple}
	a_0s_0^2 + 2a_1s_0s_1 + a_2(2s_0s_2 + s_1^2) = a_1s_0^2 + 2a_2s_0s_1 = 0.
\end{equation}
Since we assumed that $\ch \FF_q \neq 2$, there are two geometric points in $\PP(V^*_2)$ satisfying this.  If $s_0 = 0$, then we have
\[
    (s_0: s_1: s_2) = (0: 0: 1),
\]
and this point corresponds to the singular quadratic $y^2$.
If $s_0 \neq 0$, we have
\[
    \begin{aligned}
        (s_0 : s_1: s_2) = (8a_2^2: -4a_1a_2: 3a_1^2 - 4a_0a_2).
    \end{aligned}
\]
The binary quadratic $F$ corresponding to this point can be singular or not, depending on the coefficients $a_0, a_1, a_2$.  Thus, we have
\[
    - \sum_{\substack{[f_2] \in \PP(V^*_2) \\ f_2^2 \in w^\bot}} \varepsilon(f_2) = - \varepsilon(F) \in \{0, \pm 1\}.
\]

\subsubsection{The case of Waring type $\langle 1^2, 1 \rangle$}\label{sect:1^21}
By changing coordinates, we may assume that $w$ is annihilated by $xy^2$ and
\[
    w(x,y) = a_0(x^5)^* + a_1(x^4y)^* + b(y^5)^*
\]
with $a_1b \neq 0$.
The equation \eqref{eq:21} is written as
\begin{equation}\label{eq:21double}
	a_0s_0^2 + 2a_1s_0s_1 = a_1s_0^2 + bs_2^2 = 0.
\end{equation}
This defines three geometric points in $\PP(V^*_2)$, namely,
\[
    \begin{aligned}
	P_0 &= (0:1:0),\\
	P_+ &= \left( 2a_1b: -a_0b: 2a_1\sqrt{-a_1b} \right), \quad \\
	P_- &= \left( 2a_1b: -a_0b: -2a_1\sqrt{-a_1b} \right).
    \end{aligned}
\]
The first point, $P_0$, corresponds to the nonsingular form $xy$, and its contribution is
\[
    \varepsilon(xy) = 1.
\]
The other points, $P_+$ and $P_-$, need not be $\FF_q$-rational. Even when they are defined over $\FF_q$, the corresponding quadratic forms $F_+$ and $F_-$ may be split, inert, or ramified.
Thus, we have
\[
    - \sum_{\substack{[f_2] \in \PP(V^*_2) \\ f_2^2 \in w^\bot}} \varepsilon(f_2) \in \{0, \pm 1, -2, -3\}.
\]

\subsubsection{The other cases}\label{sect:other}
Since each summand is in $\{0, \pm 1\}$, 
we have
\[
    \Bigg| - \sum_{\substack{[f_2] \in \PP(V^*_2) \\ f_2^2 \in w^\bot}} \varepsilon(f_2) \Bigg| 
    \le  \sum_{\substack{[f_2] \in \PP(V^*_2) \\ f_2^2 \in w^\bot}} \left| \varepsilon(f_2) \right| \le \# \{[f_2] \in \PP(V_2^*) \mid f_2^2 \in w^\bot\}.
\]
Thus it is enough to count $[f_2] \in \PP(V^*_2)$ such that $f_2^2 \in w^\bot$.
We would like to show that 
\[
    \# \{[f_2] \in \PP(V_2^*)(\overline{\FF_q}) \mid f_2^2 \in w^\bot\} \le 4.
\]
Since we consider it geometrically, we may assume that $w$ is of Waring type $\langle 1, 1, 1\rangle$.
By changing coordinates, we only have to consider the case when $w$ is annihilated by $xy(x-y)$.
In this case, we write
\[
    w(x,y) = a(x^5)^* + b(y^5)^* + c((x^5)^* + (x^4y)^* + (x^3y^2)^* + (x^2y^3)^* + (xy^4)^* + (y^5)^*)
\]
with $abc \neq 0$.
The equation \eqref{eq:21} becomes
\[
	\begin{aligned}
	    as_0^2 - c(s_0 + s_1 + s_2)^2 = bs_2^2 - c(s_0 + s_1 + s_2)^2 = 0.
	\end{aligned}
\]
Since these equations have no common irreducible components, they define a complete intersection of conics in $\PP(V^*_2)$ and hence have at most four geometric solutions.  Thus we have
\[
    - \sum_{\substack{[f_2] \in \PP(V^*_2) \\ f_2^2 \in w^\bot}} \varepsilon(f_2) \in \{0, \pm 1, \pm 2, \pm 3, \pm 4\}.
\]

\subsection{Conclusion}
Now we summarize our computation.
\[
    \begin{aligned}
        \# \psi_{1^2, 3}^{-1}([w^\bot]_5) &= (q+1)(q^2 + q + 1) + q^3 \widetilde{N}_{1,w} \\
        \# \psi_{2^2, 1}^{-1}([w^\bot]_5) &= (q^2 + q + 1) + q \widetilde{N}_{2,w} \\
        \# \psi_{1^2, 1^2, 1}^{-1}([w^\bot]_5) &= (q+1)^2 + q \widetilde{N}_{3,w}
    \end{aligned}
\]
For $\widetilde{N}_{i, w} (1 \le i \le 3)$, we have the following table:
\begin{table}[h]
\begin{tabular}{c||c|c|c}
    Waring type of $w$ & $\widetilde{N}_{1^2, 3:w}$ & $\widetilde{N}_{2^2, 1:w}$ & $\widetilde{N}_{1^2, 1^2, 1:w}$ \\ \hline
    $\emptyset$ & $q+1$ & $q^2 + q + 1$ & $(q+1)^2$ \\ \hline
    $\langle 1 \rangle$ & $1$ & $q+1$ & $2q+1$ \\ \hline
    $\langle 1^2 \rangle$ & $1$ & $q+1$ & $2q+1$ \\ 
    $\langle 1, 1 \rangle$ & $0$ & $1$ & $2$ \\ 
    $\langle 2 \rangle$ & $0$ & $1$ & $0$ \\ 
    others & $0$ & $\widetilde{N}_{2^2, 1:w}$ & $\widetilde{N}_{1^2, 1^2, 1:w}$ \\\hline
\end{tabular}
\end{table}

In particular, we complete the computation of $N_w$:
\[
    \begin{aligned}
        N_0 &= q^4 + 2q^3 + q^2 + q + 1, \\
        N_w &= q^3 + 2q^2 + q + 1 + q^3 \widetilde{N}_{1^2,3:w} 
        + q (\widetilde{N}_{2^2, 1:w} - \widetilde{N}_{1^2, 1^2, 1: w}).
    \end{aligned}
\]
Hence we have
\[
    \begin{aligned}
        q^{6}\widehat{\Phi}(w) &= 1 + qN_w - N_0 \\
        &= 1 + q(q^3 + 2q^2 + q + 1) + q^4 \widetilde{N}_{1^2, 3:w} \\
        &\quad
        + q^2 (\widetilde{N}_{2^2, 1:w} - \widetilde{N}_{1^2, 1^2, 1: w}) - (q^4 + 2q^3 + q^2 + q + 1) \\
        &= q^4 \widetilde{N}_{1^2, 3:w} + q^2 (\widetilde{N}_{2^2, 1:w} - \widetilde{N}_{1^2, 1^2, 1:w}).
    \end{aligned}
\]
Using our estimate on the catalecticant rank 3 cases, we obtain the main result:

\begin{thm}\label{thm:ExpQuint}
	Let $V_5^*$ be the space of binary quintic forms over $\FF_q$ with $\ch \FF_q = p > 2$, and $w \in V_5 = (V_5^*)^*$. Then,
	\[
		\widehat{\Phi}(w) = \begin{cases}
			q^{-1} + q^{-2} - q^{-3} & (w=0) \\
			q^{-2} - q^{-3} & ($w$: \text{Waring type $\langle 1 \rangle$ or $\langle 1^2 \rangle$}) \\
			C(w) q^{-4} & (\text{other cases}),
		\end{cases}
	\]
	where we have
	\[
		C(w) = -\sum_{\substack{f_2 \in \PP(V^*_2) \\ f_2^2 \in w^\bot}} \varepsilon(f_2) \in \begin{cases}
			\{-1\} & (\text{Waring type of $w$ is $\langle 1, 1 \rangle$}) \\
			\{1\} & (\text{Waring type of $w$ is $\langle 2 \rangle$}) \\
			\{0, \pm 1\} & (\text{Waring type of $w$ is $\langle 1^3 \rangle$}) \\
			\{0, \pm 1, -2, -3\} & (\text{Waring type of $w$ is $\langle 1^2, 1 \rangle$}) \\
			\{0, \pm 1, \pm 2, \pm 3, \pm 4\} & (\text{other Waring types}).
		\end{cases}
	\]
\end{thm}
\begin{proof}[Proof of Theorem \ref{thm:Quint} if $\ch \FF_q > 2$]
Assume that $\Cat_{2,2}(w)$ does not vanish.
By Theorem \ref{thm:CI} (ii) and Theorem \ref{thm:ExpQuint}, we have
\[
    \widehat{\Phi}(w) = C(w)q^{-4}
\]
with $|C(w)| \le 4$.
\end{proof}

\subsection{Examples}
We give examples realizing each value of $C(w)$ in Theorem \ref{thm:ExpQuint}. This shows that the lists in Theorem \ref{thm:ExpQuint} are optimal.

\subsubsection{Waring type $\langle 1^3 \rangle$}
Over $\FF_7$, we have the following examples:
\begin{itemize}
    \item $C(w) = -1$: For $w = (x^4y)^* + (x^3y^2)^*$, the quadrics $f_2 \in V_{2^2, 1:w}$ are
    \[
        y^2, \quad 8x^2 - 4xy + 3y^2 = (x-y)(x-3y)
    \]
    This gives an example with $C(w) = -1$.
    \item $C(w) = 0$: For $w = -2(x^5)^* + (x^4y)^* + (x^3y^2)^*$, the computation in Subsubsection \ref{sect:1^3} gives the following quadrics $f_2 \in V_{2^2, 1:w}$:
    \[
        y^2, \quad 8x^2 - 4xy + 11y^2 = (x-2y)^2.
    \]
    This gives an example with $C(w) = 0$.
    \item $C(w) = 1$: For $w = (x^5)^* + (x^4y)^* + (x^3y^2)^*$, the quadrics $f_2 \in V_{2^2, 1:w}$ are
    \[
        y^2, \quad 8x^2 - 4xy - y^2.
    \]
    This gives an example with $C(w) = 1$.
\end{itemize}

\subsubsection{Waring type $\langle 1^2, 1 \rangle$}
\begin{itemize}
    \item For $w = 2(x^5)^* + (x^4y)^* - (y^5)^*$ over $\FF_{19}$, the computation in Subsubsection \ref{sect:1^21} gives the following quadrics $f_2 \in V_{2^2, 1:w}$:
    \[
        \begin{aligned}
        &xy, \\ &-2x^2 + 2xy - 2y^2 = -2(x+7y)(x+11y), 
        \\ &-2x^2 + 2xy + 2y^2 = -2(x+4y)(x-5y)
        \end{aligned}
    \]
    This gives an example with $C(w) = -3$.
    \item For $w = 3(x^5)^* + (x^4y)^* - (y^5)^*$ over $\FF_{7}$, the quadrics $f_2 \in V_{2^2, 1:w}$ are
    \[
        \begin{aligned}
        &xy, \\ &-2x^2 + 3xy + 2y^2 = -2(x-2y)(x-3y), 
        \\ &-2x^2 + 3xy - 2y^2 = -2(x+y)^2
        \end{aligned}
    \]
    This gives an example with $C(w) = -2$.
    \item For $w = (x^4y)^* - (y^5)^*$ over $\FF_{7}$, the quadrics $f_2 \in V_{2^2, 1:w}$ are
    \[
        \begin{aligned}
        &xy, \\ &-2x^2 + 2y^2 = -2(x-y)(x+y), 
        \\ &-2x^2 - 2y^2 = -2(x^2 + y^2)
        \end{aligned}
    \]
    This gives an example with $C(w) = -1$.
    \item For $w = 4(x^5)^* + (x^4y)^* - (y^5)^*$ over $\FF_{11}$, 
    the quadrics $f_2 \in V_{2^2, 1:w}$ are
    \[
        \begin{aligned}
        &xy, \\ &-2x^2 + 4xy + 2y^2 = -2(x^2 - 2xy - y^2), 
        \\ &-2x^2 + 4xy - 2y^2 = -2(x-y)^2
        \end{aligned}
    \]
    This gives an example with $C(w) = 0$.
    \item For $w = (x^5)^* + (x^4y)^* - (y^5)^*$ over $\FF_{7}$, 
    the quadrics $f_2 \in V_{2^2, 1:w}$ are
    \[
        \begin{aligned}
        &xy, \\ &-2x^2 + xy + 2y^2 = -2(x^2 + 3xy - y^2), 
        \\ &-2x^2 + xy - 2y^2 = -2(x^2 + 3xy + y^2)
        \end{aligned}
    \]
    This gives an example with $C(w) = 1$.
\end{itemize}

\subsubsection{Waring type $\langle 1, 1, 1 \rangle$}
We only treat the case when $\FF_q = \FF_p$ and 
\[
    w = a(x^5)^* + b(y^5)^* + c (-(x^5)^* - (x^4y)^* - (x^3y^2)^* - (x^2y^3)^* - (xy^4)^* - (y^5)^*)
\]
with $abc \neq 0$.
By the computation in Section \ref{sect:other}, the square of a quadric $f_2 = s_0x^2 + s_1xy + s_2y^2$ is apolar to $w$ if
\[
    as_0^2 - c(s_0 + s_1 + s_2)^2 = bs_2^2 - c(s_0 + s_1 + s_2)^2 = 0.
\]
This has nontrivial solutions if and only if $a, b, c$ are in the same class in $\FF_q^\times / \FF_q^{\times 2}$.
Multiplying a constant if necessary, 
we may assume $a = 1/\alpha^2, b = 1/\beta^2 , c = 1/\gamma^2$.
Then the quadrics are represented by
\[
    \alpha (y^2 - xy) + \gamma xy + \beta (x^2 - xy).
\]
Its discriminant is
\[
    D(\alpha, \beta, \gamma) \coloneqq \alpha^2 + \beta^2 + \gamma^2 - 2 \alpha \beta - 2 \beta \gamma - 2 \gamma \alpha,
\]
and thus
\[
    \begin{aligned}
        C(w) = -\left( \frac{D(\alpha, \beta, \gamma)}{p} \right) - \left( \frac{D(\alpha, \beta, -\gamma)}{p} \right) - \left( \frac{D(\alpha, -\beta, \gamma)}{p} \right) - \left( \frac{D(\alpha, -\beta, -\gamma)}{p} \right),
    \end{aligned}
\]
where $\left( \frac{\cdot}{p} \right)$ is the Legendre symbol.

The following table gives $D(\alpha, \beta, \gamma)$ and $C(w)$ for specific choices of $\alpha, \beta, \gamma$, where
\[
    w = \frac{1}{\alpha^2}(x^5)^* + \frac{1}{\beta^2}(y^5)^* + \frac{1}{\gamma^2} (-(x^5)^* - (x^4y)^* - (x^3y^2)^* - (x^2y^3)^* - (xy^4)^* - (y^5)^*)
\]
is defined over $\FF_q$.  The symbol $+$ indicates a nonzero square, $-$ a nonsquare, and $0$ the zero element of $\FF_q$.

\begin{tabular}{c|c||cccc|c}
    $\alpha, \beta, \gamma$ & $q$ & $D(\alpha, \beta, \gamma)$ & $D(\alpha, \beta, -\gamma)$ & $D(\alpha, -\beta, \gamma)$ & $D(\alpha, -\beta, -\gamma)$ & $C(w)$ \\
    \hline
    $1, 2, 3$ & & $-8$ & $28$ & $24$ & $12$ & \\ \hline
    & 337 & $+$ & $+$ & $+$ & $+$ & $-4$ \\
    & 19 & $+$ & $+$ & $+$ & $-$ & $-2$ \\
    & 11 & $+$ & $-$ & $-$ & $+$ & $0$ \\
    & 13 & $-$ & $-$ & $-$ & $+$ & $2$ \\
    & 7 & $-$ & $0$ & $-$ & $-$ & $3$ \\
    & 79 & $-$ & $-$ & $-$ & $-$ & $4$ \\
    \hline
    $1,3,4$ & & $-12$ & $52$ & $48$ & $16$ & \\ \hline
    & 13& $+$ & $0$ & $+$ & $+$ & $-3$ \\
    \hline
    $1,3,6$ & & $-8$ & $88$ & $76$ & $28$ & \\ \hline
    & 19& $+$ & $-$ & $0$ & $+$ & $-1$ \\
    & 7& $-$ & $+$ & $-$ & $0$ & $1$ \\
    \hline
    
\end{tabular}

\appendix
\section{The characteristic two case}
In this appendix, we consider the case $\ch \FF_q = 2$.
The treatment of this case proceeds similarly,
with modifications to the computation of $\widetilde{N}_{1^2, 3: w}$ and $\widetilde{N}_{2^2, 1: w} - \widetilde{N}_{1^2, 1^2, 1: w}$.

First, we treat $\widetilde{N}_{1^2, 3: w}$.
Recall that the number is 
\[
    \widetilde{N}_{1^2, 3:w} = \#V_{1^2, 3:w} = \#\{ [f_1] \in \PP(V^*_1) \mid f_1^2 \in w^\bot \}.
\]
This description and Lemma \ref{lem:113} remain valid in this case,
so no modification is needed.

Next, we treat $\widetilde{N}_{2^2, 1: w}$, $\widetilde{N}_{1^2, 1^2, 1: w}$, and their difference.
The arguments for $\rank (w) \le 2$ remain valid in characteristic two.
As in the Introduction, we use the function $\varepsilon$.
Then we obtain the same formula \eqref{eq:Sum}.
The condition $f_2^2 \in w^\bot$ implies two conditions on $f_2$:
\[
    \begin{cases}
        a_0s_0^2 + a_2 s_1^2 + a_4 s_2^2 = 0, \\
        a_1 s_0^2 + a_3 s_1^2 + a_5 s_2^2 = 0,
    \end{cases}
\]
where we write
\[
    w = \sum_{0 \le i \le 5} a_i (x^{5-i}y^i)^*, \quad 
    f_2 = s_0x^2 + s_1 xy + s_2y^2.
\]
These conditions are equivalent to linear conditions on $s_0, s_1, s_2$:
\begin{equation}\label{eq:char2}
    \begin{cases}
        \sqrt{a_0} s_0 + \sqrt{a_2} s_1 + \sqrt{a_4} s_2 = 0, \\
        \sqrt{a_1} s_0 + \sqrt{a_3} s_1 + \sqrt{a_5} s_2 = 0.
    \end{cases}
\end{equation}
If these conditions are linearly independent, we immediately obtain a similar estimate using \eqref{eq:Sum}.  In some cases they are linearly dependent and require more careful treatment.

\subsection{Waring type $\langle 1^3 \rangle$}
As in Section \ref{sect:1^3}, we may assume that $w$ is
\[
    w = a_0(x^5)^* + a_1(x^4y)^* + a_2(x^3y^2)^*
\]
with $a_2 \neq 0$.  Two conditions in \eqref{eq:char2} are linearly independent if $a_1 \neq 0$,
and then the only element $[f_2] \in V_{2^2, 1:w}$ is $[y^2]$.  This implies $\widetilde{N}_{2^2, 1:w} = \widetilde{N}_{1^2, 1^2, 1:w} = 1$.

If $a_1 = 0$ but $a_0 \neq 0$, we have
\[
    V_{2^2, 1:w} = \left\{ \left[s_0 \left( x^2 + \sqrt{a_0/a_2} xy \right) + s_1y^2\right] \Biggm| [s_0: s_1] \in \PP^1(\FF_q)\right\}.
\]
In this space, there are $q/2$ quadrics of type $\langle 1, 1 \rangle$,
$1$ quadric of type $\langle 1^2 \rangle$,
and $q/2$ quadrics of type $\langle 2 \rangle$.
Hence we have
\[
    \begin{aligned}
        \widetilde{N}_{2^2, 1:w} &= q+1, \\
        \widetilde{N}_{1^2, 1^2, 1:w} &= \frac{q}{2} \cdot 2 + 1 \cdot 1 + \frac{q}{2} \cdot 0 = q+1.
    \end{aligned}
\]

If $a_1 = a_0 = 0$, we have
\[
    V_{2^2, 1:w} = \{[s_0x^2 + s_1y^2] \mid [s_0: s_1] \in \PP^1(\FF_q)\},
\]
and the elements in the space are all of type $\langle 1^2 \rangle$.  We have
$\widetilde{N}_{2^2, 1:w} = \widetilde{N}_{1^2, 1^2, 1:w} = q+1$.

In any case, we have $\widetilde{N}_{2^2, 1:w} - \widetilde{N}_{1^2, 1^2, 1:w} = 0.$

\subsection{Waring type $\langle 1^2, 1 \rangle$}\label{sct:121char2}
As in Section \ref{sect:1^21}, we may assume that $w$ is
\[
    w = a_0(x^5)^* + a_1(x^4y)^* + a_5(y^5)^*
\]
with $a_1a_5 \neq 0$.  Two conditions in \eqref{eq:char2} are linearly independent if $a_0 \neq 0$,
and then the only element $[f_2] \in V_{2^2, 1:w}$ is $[xy]$.  This implies $\widetilde{N}_{2^2, 1:w} = 1$ and $\widetilde{N}_{1^2, 1^2, 1: w} = 2$.

If $a_0 = 0$, we have
\[
    V_{2^2, 1:w} = \left\{ \left[s_0 \left(x^2 + \sqrt{a_1/a_5} y^2 \right) + s_1xy \right] \Biggm|  [s_0: s_1] \in \PP^1(\FF_q)\right\}.
\]
In this space, there are $q/2$ quadrics of type $\langle 1, 1 \rangle$,
$1$ quadric of type $\langle 1^2 \rangle$,
and $q/2$ quadrics of type $\langle 2 \rangle$.
Hence we have
\[
    \begin{aligned}
        \widetilde{N}_{2^2, 1:w} &= q+1, \\
        \widetilde{N}_{1^2, 1^2, 1:w} &= \frac{q}{2} \cdot 2 + 1 \cdot 1 + \frac{q}{2} \cdot 0 = q+1.
    \end{aligned}
\]
We have 
\[
    \widetilde{N}_{2^2, 1: w} - \widetilde{N}_{1^2, 1^2, 1:w} = \begin{cases}
        -1 & (\text{if $a_0 \neq 0$}) \\
        0 & (\text{if $a_0 = 0$}).
    \end{cases}
\]

\subsection{Other Waring types}
We only estimate the number of points in $V_{2^2, 1: w}$.
As in Section \ref{sect:other}, we may assume that $w$ is
\[
    w = a(x^5)^* + b(y^5)^* + c((x^5)^* + (x^4y)^* + (x^3y^2)^* + (x^2y^3)^* + (xy^4)^* + (y^5)^*)
\]
with $abc \neq 0$.
In this case, two conditions in \eqref{eq:char2} are linearly independent,
and there is exactly one quadric $[f_2] \in V_{2^2, 1:w}$.
Thus we have
\[
    |\widetilde{N}_{2^2, 1: w} - \widetilde{N}_{1^2, 1^2, 1:w}| 
    \le |\widetilde{N}_{2^2, 1: w}| = 1.
\]

\subsection{Conclusion in the case $\ch \FF_q = 2$}
Summarizing, we obtain the following:
\begin{thm}\label{thm:ExpQuintchar2}
    Let $\ch \FF_q=2$, and $w \in V_5$. Then,
	\[
		\widehat{\Phi}(w) = \begin{cases}
			q^{-1} + q^{-2} - q^{-3} & (w=0) \\
			q^{-2} - q^{-3} & (\text{Waring type $\langle 1 \rangle$ or $\langle 1^2 \rangle$}) \\
			C(w) q^{-4} & (\text{other cases}),
		\end{cases}
	\]
	where we have
	\[
		C(w) = - \sum_{\substack{f_2 \in \PP(V^*_2) \\ f_2^2 \in w^\bot}} \varepsilon(f_2) \in \begin{cases}
			\{-1\} & (\text{Waring type of $w$ is $\langle 1, 1 \rangle$}) \\
			\{1\} & (\text{Waring type of $w$ is $\langle 2 \rangle$}) \\
			\{0\} & (\text{Waring type of $w$ is $\langle 1^3 \rangle$}) \\
            \{0, -1\} & (\text{Waring type of $w$ is $\langle 1^2, 1 \rangle$}) \\
			\{0, \pm 1\} & (\text{other Waring types})
		\end{cases}
	\]
\end{thm}
\begin{proof}[Proof of Theorem \ref{thm:Quint} if $\ch \FF_q = 2$]
Assume that $\Cat_{2,2}(x,y\colon w)$ does not vanish.
By Theorem \ref{thm:CI} (ii) and Theorem \ref{thm:ExpQuintchar2}, we have
\[
    \widehat{\Phi}(w) = C(w)q^{-4}
\]
with $|C(w)| \le 1$.
\end{proof}

\subsection{Examples}
Over $\FF_2$, the integer $C(w)$ attains every value listed in Theorem \ref{thm:ExpQuintchar2}.
Since the examples with Waring type $\langle 1^2, 1\rangle$ can be obtained from Section \ref{sct:121char2}, we only give examples when the Waring type of $w$ is $\langle 1, 1, 1 \rangle, \langle 2, 1\rangle$ or $\langle 3 \rangle$.
\begin{itemize}
    \item If we take $a = b = c = 1$, then
    \[
    w = (x^4y)^* + (x^3y^2)^* + (x^2y^3)^* + (xy^4)^*
    \]
    and the unique $[f_2] \in V_{2^2, 1:w}$ is $x^2 + xy + y^2$.
    This is irreducible over $\FF_2$ and we have $\varepsilon(f_2) = -1$.
    \item We take
    \[
        w = (x^5)^* + (x^4y)^* + (x^3y^2)^* + (xy^4)^* + (y^5)^*.
    \]
    Its Waring type is $\langle 2, 1\rangle$, and the unique $[f_2] \in V_{2^2, 1:w}$ is $x^2 + y^2$.
    We also have $\varepsilon(f_2) = 0$.
    \item We take
    \[
        w = (x^5)^* + (x^2y^3)^* + (y^5)^*.
    \]
    Its Waring type is $\langle 3 \rangle$, and the unique $[f_2] \in V_{2^2, 1:w}$ is $xy + y^2$.
    This is reducible over $\FF_2$ and $\varepsilon(f_2) = 1$.
\end{itemize}

\end{document}